\documentclass[12pt,english]{article}

\usepackage{amsfonts,amsmath,amssymb,epsfig,epsf,psfrag}
\usepackage[english]{babel}
\usepackage{enumerate}
\usepackage{url}
\usepackage{color}
\usepackage[utf8]{inputenc}
\usepackage{graphicx}
\usepackage[T1]{fontenc}
\usepackage{amsthm}
\usepackage{multimedia}
\usepackage{changebar}

\theoremstyle{plain}
\newtheorem{thrm}{Theorem}[section]
\newtheorem{lmm}[thrm]{Lemma}
\newtheorem{crllr}[thrm]{Corollary}
\newtheorem{prpstn}[thrm]{Proposition}

\theoremstyle{definition}
\newtheorem{dfntn}[thrm]{Definition}

\theoremstyle{plain}
\newenvironment{acknowledgement}{\par\addvspace{17pt}\small\rmfamily\noindent}{\par\addvspace{6pt}}

\def\d{\displaystyle}
\def\R{\mathbf{R}}

\def\E{\mathbf{E}}
\def\P{\mathbf{P}}

\def\M{\mathcal{M}}
\def\NB{\mathcal{NB}}
\def\barsmh{\bar{s}_{\hat{m}}}
\def\barsm{\bar{s}_m}
\def\barsmp{\bar{s}_{m'}}
\def\barga{\bar{\gamma}}
\def\hatsmh{\hat{s}_{\hat{m}}}
\def\hatsm{\hat{s}_m}
\def\hatsmp{\hat{s}_{m'}}
\def\mh{\hat{m}}

\def\Omf{\Omega_{m_f}(\epsilon)}

\def\OO1{\Omega_1(\xi)}
\def\mO2{\Omega_2(\xi)}
\def\[[{[\![}
\def\O{\mathbf{1}_{\Omega(\epsilon,\xi)}}

\begin{document}

\title{\bf Segmentation of the Poisson and negative binomial rate models: a penalized estimator}

\author{A. Cleynen$^{1}$ and E. Lebarbier$^{1}$}
\date{}
\maketitle \noindent $^{1}${\it AgroParisTech/ INRA MIA 518, 16 rue
Claude
Bernard, 75231 Paris Cedex 05, France. \\
E-mail: alice.cleynen@agroparistech.fr\\
E-mail: emilie.lebarbier@agroparistech.fr} \\

\date{\today}
\maketitle

\begin{abstract}
We consider the segmentation problem of Poisson and negative
binomial (i.e. overdispersed Poisson) rate distributions. In
segmentation, an important issue remains the choice of the number of
segments. To this end, we propose a penalized $\log$-likelihood
estimator where the penalty function is constructed in a
non-asymptotic context following the works of L. Birg\'e and P.
Massart. The resulting estimator is proved to satisfy an oracle
inequality. The performances of our criterion is assessed using
simulated and real datasets in the RNA-seq data analysis context.
\end{abstract}

\vspace{0.2cm} \noindent
\section*{Introduction}

We consider a multiple change-point detection setting for count
datasets, which can be written as follows: we observe a finite
sequence $\{y_t\}_{t \in \{1,\dots,n\}}$ realisation of independent
variables $Y_t$. These variables are supposed to be drawn from a
probability distribution $\mathcal{G}$ which depends on a set of
parameters. Here two types of parameters are distinguished:
$$
Y_t\sim \mathcal{G}(\theta_t,\phi)=s(t), \; 1\leq t\leq n,
$$
where $\phi$ is a constant parameter while the $\theta$s are
point-specific. In many contexts, we might want to consider that the
$\theta$s are piece-wise constant and so subject to an unknown
number $K-1$ of abrupt changes (for instance with climatic or
financial data). Thus, we want to assume the existence of partition
of $\{1,\dots,n\}$ into $K$ segments within which the observations
follow the same distribution and between which observations have
different distributions, i.e. $\theta$ is constant within a segment
and differ from a segment to another. A motivating example is
sequencing data analysis. For instance, the output of RNA-seq
experiments is the number of reads (i.e. short portions of the
genome) which first position maps to each location of a genome of
reference. Supposing that we dispose of such a sequence, we expect
to observe a stationarity in the amount of reads falling in
different areas of the genome: expressed genes, intronic regions,
etc. We wish to localize those regions that are biologically
significant. In our context, we consider for $\mathcal{G}$ the
Poisson and negative binomial distributions,
adapted to RNA-seq experiment analysis \cite{Risso_norma}. \\
Change-point detection problems are not new and many methods have
been proposed in the literature. For count data-sets,
\cite{BraunMuller1998} provide a detailed bibliography of methods in
the particular case of the segmentation of the DNA sequences that
includes Bayesian approaches, scan statistics, likelihood-ratio
tests, binary segmentation and numerous other methods such as
penalized contrast estimation procedures. In a Bayesian framework,
\cite{BCG00} proposes to use an exact "ICL" criterion for the choice
of $K$, while its approximation is computed in the constrained HMM
approach of \cite{luong2012fast}.  In this paper, we consider a
penalized contrast estimation method which consists first, for every
fixed $K$, in finding the best segmentation in $K$ segments by
minimizing the contrast over all the partitions with $K$ segments,
and then in selecting a convenient number of segments $K$ by
penalizing the contrast. Choosing the number of segments, i.e.
choosing a "good" penalty, is a crucial issue and not so easy. The
most basic examples of penalty are the Akaike Information Criterion
(AIC \cite{AIC}) and the Bayes Information Criterion (BIC
\cite{Yao88}) but these criteria are not well adapted in the
segmentation context and tend to overestimate the number of
change-points (see \cite{BMMinPen, zhang_modified_2007} for
theoretical explanations). In this particular context, some modified
versions of these criteria have been proposed. For instance,
\cite{zhang_modified_2007,BraunBraunMuller2000}  have proposed
modified versions of the BIC criterion (shown to be consistent) in
the segmentation of Gaussian processes and DNA sequences
respectively. However, these criteria are based on asymptotic
considerations. In the last years there has been an extensive
literature influenced by \cite{BM97,BBM} introducing non-asymptotic
model selection procedures, in the sense that the size of the models
as well as the size of the list of models are allowed to be large
when $n$ is large. This penalized contrast procedure consists in
selecting a model amongst a collection such that its performance is
as close as possible to that of the best but unreachable model in
terms of risk. This approach has been now considered in various
function estimation contexts. In particular, \cite{AkakpoESAIM}
proposed a penalty for estimating the density of independent
categorical variables in a least-squares framework, while
\cite{RBPoisson,BirgePoisson}, or \cite{BaraudBirge},
focused on the estimation of the density of a Poisson process. \\
When the number of models is large, as in the case of an exhaustive
search in segmentation problem, it can be shown that penalties which
only depend on the number of parameters of each model, as for the
classical criteria, are theoretically (and also practically) not
adapted. This was suggested by
\cite{lebarbier_detecting_2005,BMMinPen} who show that the penalty
term needs to be well defined, and in particular needs to depend on
the complexity of the list of models, i.e. the number of models
having the same dimension. For this reason, following the work of
\cite{BM97} and in particular \cite{Cast00} in the density
estimation framework, we consider a penalized $\log$-likelihood
procedure to estimate the true distribution $s$ of a Poisson or
negative binomial-distributed sequence $\mathbf{y}$. We prove that,
up to a $\log n$ factor, the resulting estimator satisfies an oracle
inequality.

The paper is organized as follows. The general framework is
described in Section \ref{framework}. More precisely, we present our
proposed penalized maximum-likelihood estimator, the form of the
penalty and give some non-asymptotic risk bounds for the resulting
estimator. The studies of the two considered models (Poisson and
negative binomial) are done in parallel along the paper. Some
exponential bounds are derived in Section \ref{IntermediateResults}.
A simulation study is performed to compare our proposed criterion
with others and an application to the segmentation of RNA-seq data
illustrates the procedure in Section \ref{Appli}. The proof of the
main result is given in Section \ref{dem-theo} for which the proofs
of some intermediate results are given in the Appendix
\ref{Appendices}.

\section{Model Selection Procedure} \label{framework}
\subsection{Penalized maximum-likelihood estimator}

Let us denote by $m$ a partition of $\[[1,n]\!]$, $m = \{
\[[1,\tau_1 \[[, \[[\tau_1, \tau_2\[[, \dots, \[[ \tau_k, n ]\!] \}$
and by $\mathcal{M}_n$ a set of partitions of $\[[1,n]\!]$. In our
framework we want to estimate the distribution $s$ defined by
$s(t)=\mathcal{G}(\theta_t,\phi), \, 1\leq t \leq n $, and we
consider the two following models:
  \begin{eqnarray*}
  \begin{array}{lll}
  \mathcal{G}(\theta_t,\phi) & =  \mathcal{P}(\lambda_t) & (\mathcal{P}) \\
  \mathcal{G}(\theta_t,\phi) & =  \NB(p_t,\phi) \quad \quad & (\NB)
  \end{array}
  \end{eqnarray*}

In the ($\NB$) case, we suppose that the over-dispersion
parameter $\phi$ is known. We define the collection of models :

  \begin{dfntn}
The collection of models associated to partition $m$ is
$\mathcal{S}_m$ the set of distribution of sequences of length $n$
such that for each element $s_m$ of $\mathcal{S}_m$, for each
segment $J$ of $m$, and for each $t$ in $J$, $s_m(t) =
\mathcal{G}(\theta_J,\phi)$:
    \begin{eqnarray*}
      \mathcal{S}_m =\left\{s_m \;| \;  \forall J \in m, \; \forall t \in J, \; s_m(t) = \mathcal{G}(\theta_J,\phi)
      \right\}.
    \end{eqnarray*}
\end{dfntn}

We shall denote by $|m|$ the number of segments in partition $m$,
and by $|J|$ the length of segment $J$.

We consider the log-likelihood contrast $\gamma(u) = \sum_{t=1}^n
-\log \P_u(Y_t)$, namely respectively for $u(t)=\mathcal{P}(\mu_t)$ and $u(t)=\NB(q_t,\phi)$,
  \begin{eqnarray*}
  \begin{array}{lll}
    \gamma(u) &= \sum_{t=1}^n \mu_t - Y_t \log(\mu_t) + \log(Y_t!), & (\mathcal{P}) \\
    \gamma(u) &= \sum_{t=1}^n -\phi \log q_t - Y_t \log(1-q_t)-\log\left(\frac{\Gamma(\phi+Y_t)}{\Gamma(\phi)Y_t!}\right). \quad \quad& (\NB)
    \end{array}
  \end{eqnarray*}

Then the minimal contrast estimator $\hatsm$ of $s$ on the collection $\mathcal{S}_m$ is
  \begin{eqnarray} \label{esti}
    \hatsm = \arg \min_{u\in \mathcal{S}_m} \gamma(u),
  \end{eqnarray}

so that, noting $\bar{Y}_J = \dfrac{\sum_{t\in J} Y_t}{|J|}$, for all $J \in m$ and $t \in J$
\begin{eqnarray} \label{shat}
 \hatsm(t) = \mathcal{P}(\bar{Y}_J) \, \text{for } (\mathcal{P}) \quad \text{and} \; \hatsm(t) =  \NB\left(\dfrac{\phi}{\phi + \bar{Y}_J},\phi\right) \, \text{for } (\NB).
\end{eqnarray}

Therefore, for each partition $m$ of $\M_n
$ we can obtain the best estimator $\hatsm$ as in equation
(\ref{shat}), and thus define a collection of estimators
$\left\{(\hatsm)_{m\in \M_n}\right\} $. Ideally, we would wish to
select the estimator $\hat{s}_{m(s)}$ amongst this collection with the minimum given risk. In the
log-likelihood framework, it is natural to consider the
Kullback-Leibler risk, with $K(s,u)=\E\left[\gamma(u)-\gamma(s)
\right]$. In the following we note $\E$ and $\P$ the expectation and
the probability under the true distribution $s$ respectively
(otherwise the underlying distribution is mentioned). In our models,
the Kullback-Leibler between distributions $s$ and $u$ can be
developed into
  \begin{eqnarray*}
  \begin{array}{lll}
 K(s,u) &= \sum_{t=1}^n \left(\mu_t - \lambda_t -\lambda_t \log\dfrac{\mu_t}{\lambda_t} \right), & (\mathcal{P}) \\
 K(s,u) &= \phi \sum_{t=1}^n \log \left(\dfrac{p_t}{q_t} \right) + \dfrac{1-p_t}{p_t}\log \left(\dfrac{1-p_t}{1-q_t} \right). \quad \quad & (\NB)
 \end{array}
\end{eqnarray*}

Unfortunately, minimizing this risk requires the knowledge of the
true distribution $s$, and is unreachable. We will therefore want to
consider the estimator $\hatsmh $ where $\mh$ minimizes $
\gamma(\hatsm) + pen(m)$ for a well-chosen function $pen$ (depending
on the data). By doing so, we hope to select an estimator $\hatsmh$
whose risk is as close as possible to the risk of
$\hat{s}_{m(s)}=\arg\min_{m\in\mathcal{M}_n} \E_s[K(s,\hat{s}_m)]$ in the sense that
  \begin{eqnarray*}
    \E[K(s,\hatsmh)] \leq C \, \E[K(s,\hat{s}_{m(s)})],
  \end{eqnarray*}
where $C$ is a nonnegative constant hopefully close to $1$. We therefore introduce the following definition: \newline

  \begin{dfntn} \label{def-hatsm}
    Let $\mathcal{M}_n$ be a collection of partitions of $[\![1,n]\!]$ constructed on a partition $m_f$ (i.e.  $m_f$ is a refinement of every $m$ in $\mathcal{M}_n$). Given a nonnegative, increasing in the size of $m$ penalty function $pen$: $\M_n \rightarrow \R_+$, and choosing
        \begin{eqnarray*}
      \hat{m} = \arg \min_{m\in \mathcal{M}_n} \{\gamma(\hatsm) + pen(m) \},
    \end{eqnarray*}
    we define the penalized maximum-likelihood estimator as $\hatsmh$.
  \end{dfntn}

In the following Section we provide a choice of penalty function,
and show that the resulting estimator satisfies an oracle
inequality.

\subsection{Choice of the penalty function} \label{Main}
\subsubsection*{Main result}

The following result shows that for an appropriate choice of the penalty function, we have a non-asymptotic risk bound for the penalized maximum-likelihood estimator.

  \begin{thrm} \label{theo}
    Let $\mathcal{M}_n$ be a collection of partitions constructed on a partition $m_f$ such that there exist absolute positive constants $\rho_{min}$, $\rho_{max}$ and $\Gamma$ satisfying:
    \begin{itemize}
      \item $\forall t, \rho_{min} \leq \theta_t \leq \rho_{max}$ and
      \item $\forall J \in m_f, |J|\geq \Gamma (\log(n))^2$.
    \end{itemize}
    Let $(L_m)_{m\in \mathcal{M}_n}$ be some family of positive weights  satisfying
    \begin{eqnarray}
      \Sigma = \sum_{m\in\mathcal{M}_n} \exp(-L_m|m|) < +\infty . \label{weights}
    \end{eqnarray}
    Let $\beta > 1/2$ in the Poisson case, $\beta > 1/4$ in the negative binomial case. If for every $m \in \mathcal{M}_n$
    \begin{eqnarray}
      pen(m)\geq \beta |m| \left(1 + 4\sqrt{L_{m}} \right)^2,
\label{penalty}
    \end{eqnarray}
    then
    \begin{eqnarray*}
      \E\left[h^2(s,\hatsmh) \right] &\leq & C_{\beta} \inf_{m\in\mathcal{M}_n} \left\{K(s,\barsm) +pen(m)\right\} + C(\phi,\Gamma,\rho_{min},\rho_{max},\beta,\Sigma),
    \end{eqnarray*}
    with $C_\beta=\dfrac{ (16 \beta)^{1/3}}{(2 \beta)^{1/3}-1}$ in model $(\mathcal{P})$  and $C_\beta=\dfrac{ (4 \beta)^{1/3}}{(4 \beta)^{1/3}-1}$ in model
    $(\NB)$.
  \end{thrm}
We note $h^2(s,u)$ the squared Hellinger distance between
distribution $s$ and $u$ and $\barsm$ is the projection of $s$ onto
the collection $\mathcal{S}_m$ according to the Kullback-Leibler
distance. The proof of this Theorem is given in Section
\ref{dem-theo}.
\newline

Denoting $\barsm = \arg \min_{u \in \mathcal{S}_m} K(s,u) $, we have
for $J\in m$ and $t \in J$,
  \begin{eqnarray}\label{sbar}
  \begin{array}{llll}
    \barsm(t) &= \mathcal{P}(\bar{\lambda}_J) \quad &\text{where} \; \bar{\lambda}_J = \dfrac{\sum_{t\in J} \lambda_t}{|J|} &(\mathcal{P})\\
    \barsm(t) &= \mathcal{NB}(p_J, \phi) \quad &\text{where} \;  p_J=\dfrac{|J|}{\sum_{t\in J}1/p_t}.\quad \quad & (\NB)
  \end{array}
  \end{eqnarray}

We remark that the risk of the penalized estimator
$\hat{s}_{\hat{m}}$ is treated in terms of Hellinger distance
instead of the Kullback-Leibler information. This is due to the fact
that the Kullback-Leibler is possibly infinite, and so difficult to
control. It is possible to obtain a risk bound in term of
Kullback-Leibler if we have a uniform control of $||\log
(s/\bar{s}_m)||_{\infty }$ (see \cite{Massart} for more
explanation).

\subsubsection*{Choice of the weights $\{L_m, m \in \M_n \}$.}

The penalty function depends on the family $\M_n$ through the choice
of the weights $L_m$ which satisfy (\ref{weights}). We consider for
$\M_n$ the set of all possible partitions of $[\![1,n]\!] $
constructed on a partition $m_f$ which satisfies, for all segment
$J$ in $m_f$, $|J|\geq \Gamma (\log n)^2$. Classically (see
\cite{BMGaussian}) the weights are chosen as a function of the
dimension of the model $s$, which is here $|m|$. The number of
partitions of $\M_n$ having dimension $D$ being bounded by $n
\choose{D}$, we have
\begin{eqnarray*}
\Sigma & = & \sum_{m \in \M_n} e^{L_m|m|} = \sum_{D=1}^n e^{-L_D D} Card \{m \in \M_n, |m|=D\} \\\
&\leq &  \sum_{D=1}^n {n \choose{D}} e^{-L_D D} \leq   \sum_{D=1}^n \left(\dfrac{en}{D} \right)^D e^{-L_D D}\\
&\leq &  \sum_{D=1}^n  e^{-D \left(L_D-1-\log\left(\dfrac{n}{D}\right) \right)}.\\
\end{eqnarray*}
So with the choice $L_D = 1 + \kappa + \log\left(\dfrac{n}{D}\right) $
with $\kappa>0$, condition (\ref{weights}) is satisfied. Choosing, say $\kappa=0.1$, the penalty
function can be chosen of the form
\begin{eqnarray}
pen(m)=\beta |m|\left(1 +
4\sqrt{1.1+\log{\left(\frac{n}{|m|}\right)}}\right)^2,
\end{eqnarray}
where $\beta$ is a constant to be calibrated. \newline

Integrating this penalty in Theorem \ref{theo} leads to the following control:
\begin{eqnarray} \label{almost-corollary}
      \E\left[h^2(s,\hatsmh) \right] \leq C_1 \inf_{m\in\mathcal{M}_n} \left\{K(s,\barsm) +\beta |m|\left(1 + 4\sqrt{1.1+\log{\left(\frac{n}{|m|}\right)}}\right)^2\right\} + C(\phi,\Gamma,\rho_{min},\rho_{max},\beta,\Sigma)
\end{eqnarray}
The following proposition gives a bound on the Kullback-Leibler risk associated to $\hatsm$:

  \begin{prpstn}  \label{Risk-control}
    Let $m$ be a partition of $\mathcal{M}_n$, $\hatsm$ be the minimum contrast estimator and $\barsm$ be the projection of $s$ given by equations (\ref{shat}) and (\ref{sbar}) respectively.
    Assume that there exists some positive absolute constants $\rho_{min}$, $\rho_{max}$ and $\Gamma$ such that
    $\forall t, \, \rho_{min} \leq \theta_t \leq \rho_{max}$ and
    $|J|\geq \Gamma (\log n)^2$. Then $\forall \varepsilon >0, \forall a>2$

    \begin{eqnarray*}
      K(s,\bar{s}_m)- \dfrac{C_1(\phi,\Gamma,\rho_{min},\rho_{max},\varepsilon,a)}{n^{a/2-\alpha}}+C_2(\varepsilon)|m|\leq \E[K(s,\hat{s}_m)],
    \end{eqnarray*} where $\alpha<1$ is a constant that can be expressed according to $n$, $C_2(\varepsilon)=\dfrac{1}{2} \dfrac{1-\varepsilon}{(1+\varepsilon)^2}$  in the Poisson model $(\mathcal{P})$ and $C_2(\varepsilon)=\rho_{min}^2\dfrac{(1-\varepsilon)^2}{(1+\varepsilon)^4}$ in the negative binomial model $(\NB)$.
  \end{prpstn}
The proof is given in appendix \ref{risk}. \newline

  Combining proposition \ref{Risk-control} and equation (\ref{almost-corollary}), we obtain the following oracle-type inequality:

  \begin{crllr} \label{corollary}
    Let $\mathcal{M}_n$ be a collection of partitions constructed on a partition $m_f$ such that there exist absolute positive constants $\rho_{min}$, $\rho_{max}$ and $\Gamma$ verifying:
    \begin{itemize}
      \item $\forall t, \rho_{min} \leq \theta_t \leq \rho_{max}$ and
      \item $\forall J \in m_f, |J|\geq \Gamma (\log n)^2$.
    \end{itemize}
    There exists some absolute constant $C$ such that
    \begin{eqnarray*}
      \E\left[h^2(s,\hatsmh) \right] &\leq & C \log(n) \inf_{m\in \M_n} \left\{\E[K(s,\hatsm)]\right\} + C(\phi,\Gamma,\rho_{min},\rho_{max},\beta,\Sigma).
    \end{eqnarray*}
  \end{crllr}

\section{Exponential bounds}\label{IntermediateResults}

In order to prove Theorem \ref{theo}, the general procedure in this
model selection framework (see for example \cite{BMGaussian}) is the
following: by definitions of $\hat{m}$ and $\hatsm$ (see definition
\ref{def-hatsm} and equation (\ref{esti})), we have, $\forall m \in
\mathcal{M}_n$
\begin{eqnarray*}
      \gamma(\hatsmh) + pen(\hat{m}) \leq \gamma(\hatsm) + pen(m) \leq \gamma(\barsm) + pen(m).
    \end{eqnarray*}
   Then, with  $\barga(u) = \gamma(u) -\E[\gamma(u)] $,
    \begin{eqnarray*}
      K(s,\hatsmh) \leq K(s,\barsm) + \barga(\barsm) -\barga(\hatsmh) -pen(\mh) +pen(m).
    \end{eqnarray*}
The idea is therefore to control $\barga(\barsm)
-\barga(\hat{s}_{m'})$ uniformly over $m' \in \mathcal{M}_n$. This is more complicated when dealing with different models $m$ and $m'$. Thus, following the work of \cite{Cast00}
(see proof of Theorem 3.2, also recalled in \cite{Massart}), we
propose the following decomposition
    \begin{eqnarray} \label{decomposition}
      \barga(\barsm) -\barga(\hat{s}_{m'}) = \left(\barga(\bar{s}_{m'}) -\barga(\hat{s}_{m'})\right) + \left(\barga(s) - \barga(\bar{s}_{m'})\right) + \left(\barga(\barsm)
      -\barga(s)\right),
    \end{eqnarray}
and control each term separately. The first term is the most
delicate to handle, and requires the introduction and the control of
a chi-square statistic. The main difficulty here is the non-bounded
characteristic of the objects we are dealing with. Indeed, in the
classic density estimation context such as that of \cite{Cast00},
the objects are probabilities which are bounded and so facilitate
the direct use of concentration
inequalities. \\
In our case, the chi-square statistic we introduce is denoted
$\chi^2_m$ and defined by
\begin{eqnarray}    \label{def-chi}
\chi^2_m=\chi^2(\bar{s}_m,\hat{s}_m)= \sum_{J\in m} |J|
\dfrac{(\bar{Y}_J-\bar{E}_J)^2}{\bar{E}_J},
\end{eqnarray}
where we recall that $\bar{Y}_J = \dfrac{\sum_{t\in J} Y_t}{|J|}$
and use the notation $\bar{E}_J=\frac{E_J}{|J|}$ with
$E_J=\sum_{t\in J} E_t$. Respectively for $(\mathcal{P})$ and
$(\NB)$, we have $E_t=\lambda_t$ and $E_t=\phi\frac{1-p_t}{p_t}$.
The purpose is thus to control $\chi^2_m$ uniformly over
$\mathcal{M}_n$. To this effect, we need to obtain an exponential
bound of $Y_J=\sum_{t\in J} Y_t$ around its expectation. In
Subsection \ref{ControlYJ}, we recall a result of \cite{BaraudBirge}
that we use to derive an exponential bound for $\chi^2_m$
(Subsection \ref{ControlChi2}).

\subsection{Control of $Y_J$} \label{ControlYJ}

First we recall a large deviation results established by
\cite{BaraudBirge} (lemma 3) that we apply in the Poisson and
negative binomial frameworks.
\begin{lmm}
Let $Y_1,\dots,Y_n$ be $n$ independent centered random variables.

If $\log(\E[e^{zY_i}])\leq \kappa\dfrac{z^2\theta_i}{2(1-z\tau)} $
for all $z\in [0,1/\tau[$, and $1\leq i\leq n$, then
\begin{eqnarray*}
\mathbf{P}\left[\sum_{i=1}^n Y_i \geq \left(2\kappa x\sum_{i=1}^n
\theta_i\right)^{1/2} +\tau x\right]\leq e^{-x} \ \text{for all }
x>0.
\end{eqnarray*}
If for $1\leq i\leq n$ and all $z>0$ $\log(\E[e^{-zY_i}])\leq \kappa
z^2\theta_i/2$, then
\begin{eqnarray*}
\mathbf{P}\left[\sum_{i=1}^n Y_i \leq -\left(2\kappa x\sum_{i=1}^n
\theta_i\right)^{1/2} \right]\leq e^{-x} \ \text{for all } x>0.
\end{eqnarray*}
\end{lmm}
To apply this lemma we therefore need a majoration of
$\log{\E\left[e^{z(Y_t-E_t)}\right]}$ and
$\log{\E\left[e^{-z(Y_t-E_t) }\right]}$ for $z>0$.

\subsubsection*{Poisson case.}  With $E_t=\lambda_t$,  we have:
\begin{eqnarray*}
\log{\E\left[e^{z(Y_t-\lambda_t) }\right]} &=& - z
\lambda_t+\log{\E\left[e^{z Y_J}\right]} = - z
\lambda_t+\log{e^{(\lambda_t(e^{z}-1))}} = \lambda_t (e^{z}-z-1).
\end{eqnarray*}
So
\begin{eqnarray*}
\log{\E\left[e^{z(Y_t-E_t) }\right]} = E_t (e^{z}-z-1).
\end{eqnarray*}

\subsubsection*{Negative binomial case.} In this case $E_t=
\phi\frac{1-p_t}{p_t}$ and we have
\begin{eqnarray*}
  \log \E\left(e^{z\left(Y_t-\phi\frac{1-p_t}{p_t} \right)} \right)&=& -z \phi \frac{1-p_t}{p_t}+\phi\log\frac{p_t}{1-(1-p_t)e^z} \, \text{for }z\leq -\log(1-p_t)\\
  &\leq& \phi\left[ \frac{1-p_t}{p_t}(-z)+  \frac{1-p_t}{p_t} \frac{p_t}{1-(1-p_t)e^z}(e^z-1) \right]\\
  &\leq& \phi\left[ \frac{1-p_t}{p_t}(-z)+  \frac{1-p_t}{p_t} (e^z-1) \right] \leq \phi \frac{1-p_t}{p_t}(e^z-z-1).\\
\end{eqnarray*}
So that in both cases,
\begin{eqnarray*}
\log{\E\left[e^{z(Y_t-E_t) }\right]} \leq E_t (e^{z}-z-1).
\end{eqnarray*}

Now using $e^{z}-z-1 \leq \frac{ z^2}{ 2(1-z)}$ for $z>0$ and
$e^{z}-z-1 \leq \frac{z^2}{2}$ for $z<0$, we have
\begin{eqnarray*}
\log{\E\left[e^{z(Y_t-E_t) }\right]} \leq  E_t\frac{z^2}{2(1-z)}
\quad \text{and} \quad \log{\E\left[e^{-z(Y_t-E_t) }\right]} \leq
E_t\frac{z^2}{2}
\end{eqnarray*}

Then,
\begin{equation*}
P\left[Y_J-E_J\geq \sqrt{2x E_J}+x \right] \leq e^{-x},
\end{equation*}

or
\begin{equation} \label{ControlY-lambda}
P\left[Y_J-E_J\geq x \right] \leq e^{-\frac{x^2}{2(E_J+x)}}  \ \
\text{and} \ \ P\left[|Y_J-E_J|\geq x \right] \leq  2
e^{-\frac{x^2}{2(E_J+x)}}
\end{equation}

\subsection{Exponential bound for $\chi^2_m$} \label{ControlChi2}

We first introduce the following set $\Omega_m$ defined by:
\begin{eqnarray}    \label{def-Omega}
\Omega_m(\varepsilon) =\d \bigcap_{J \in m} \left\{\left|\dfrac{Y_J}{E_J}-1
\right|\leq \varepsilon \right\},
\end{eqnarray}
for all $\varepsilon \in ]0,1[$ and all segmentations $m$ such that
each segment $J$ verifies $|J|\geq \Gamma(\log(n))^2$. This set
has a large probability since we obtain
\begin{eqnarray*}
\P(\Omega_m(\varepsilon)^C)&\leq&\sum_{J\in m}\P\left(\left|Y_J -E_J \right|>\varepsilon E_J \right) \leq 2 \sum_{J \in m}e^{-\frac{\varepsilon^2 E_J}{2(1+\varepsilon)}} \\
&\leq & 2 \sum_{J \in m} e^{-|J|\varepsilon'f(\phi,\rho_{min}) }
\leq 2|m|\exp({-\varepsilon' \Gamma f(\phi,\rho_{min})(\log(n))^2})
    \end{eqnarray*}
by applying equation (\ref{ControlY-lambda}) with $x=\varepsilon
E_J$ and where $\varepsilon'=\varepsilon^2/(2 (1+\varepsilon))$ and
$f(\phi,\rho_{min})>0 $. Thus
\begin{eqnarray} \label{ControlOmega}
      \P(\Omega_m(\varepsilon)^C)&\leq&\dfrac{C(\phi,\Gamma,\rho_{min},\varepsilon,a)}{n^a},
    \end{eqnarray}
with $a>2$. \\
The reason for introducing this set is double: in addition to enable the
 control of $\chi^2_m$ given by equation (\ref{def-chi}) on
this restricted set, it allows us to link $K(\hat{s}_m,\bar{s}_m)$
to $V^2_m$ (see (\ref{lienKV2}) for the control of the first term in
the decomposition) and so to $\chi^2_m$, relation that we use to
evaluate the risk of one model (see (\ref{link:V2Chi2})). \\

Let $m_f$ be a partition of $\mathcal{M}_n$ such that $\forall J \in
m_f, |J|\geq \Gamma (\log(n))^2$ and assume that all considered
partitions in $\mathcal{M}_n$ are constructed on this grid $m_f$.
The following proposition gives an exponential bound for
$\chi^2_{m}$ on the restricted event $\Omf$.
\begin{prpstn} \label{cont-chi2}
Let $Y_1,\ldots,Y_n$ be independent random variables with
 distribution $\mathcal{G}$ (Poisson or negative binomial
distribution). Let $m$ be a partition of $\mathcal{M}_n$ with $|m|$
segments and $\chi^2_{m}$ the statistic given by (\ref{def-chi}).
For any positive $x$, we have
\begin{eqnarray*}
\P\left[ \chi^2_{m} \mathbf{1}_{\Omf} \geq |m| + 8 (1+\varepsilon)
\sqrt{x |{m}|} +4 (1+\varepsilon) x\right] &\leq &
        e^{-x}.
\end{eqnarray*}
\end{prpstn}

\begin{proof}
As in the density estimation framework, this quantity can be
controlled using the Bernstein inequality. In
our context, noting $\chi^2_m=\sum_{J\in m} Z_J$ where
$$
Z_J=\frac{(Y_{J}-E_{J})^2}{E_{J}},
$$
we need
\begin{itemize}
\renewcommand{\labelitemi}{$\bullet$}
    \item the calculation (or bounds) of the expectation of $\chi^2_m$:

\subsubsection*{Poisson case}
$Y_{J}$ is distributed according to a Poisson distribution with parameter $\lambda_{J}$ so that
        \begin{eqnarray}
            \E\left[\chi^2_m\right] = |m|.
        \end{eqnarray}
\subsubsection*{Negative binomial case}  We have
        \begin{eqnarray*}
            \E\left[\chi^2_m\right]=\sum_{J \in m} \frac{1}{|J|} \dfrac{\sum_{t\in J}Var(Y_t)}{\phi \frac{1-p_J}{p_J}} = \sum_{J \in m} \frac{1}{|J|}
\dfrac{\sum_{t\in J}\phi
\frac{1-p_t}{p_t^2}}{\phi\frac{1-p_J}{p_J}},
        \end{eqnarray*}
        and thus
        \begin{eqnarray}
            |m|\leq \E\left[\chi^2_m\right] \leq \frac{1}{\rho_{min}}|m|.
        \end{eqnarray}

    \item an upper bound of $\sum_{J\in m} \E[Z_J^p]$. For every $p\geq 2$ we have,
    \begin{eqnarray*}
        \E\left[ Z_J^{p} \mathbf{1}_{\Omf}  \right] &=&\frac{1}{E_J^{p}} \int_{0}^{+\infty } 2p \  x^{2p-1}P\left[ \left\{ |Y_J-E_J|\geq x\right\} \cap \Omf \right] dx \\
        &\leq &\frac{1}{E_J^{p}} \int_{0}^{\varepsilon E_J} 2p \ x^{2p-1}P\left[|Y_J-E_J|\geq x \right] dx\text{.}
    \end{eqnarray*}
\end{itemize}

Using equation (\ref{ControlY-lambda}) and since $x \leq \varepsilon E_J$, we obtain the exponential
bound $P\left[|Y_J-E_J | \geq x \right] \leq  2
e^{-\frac{x^2}{2E_J(1+\varepsilon)}}.$

Therefore

\begin{eqnarray*}
\E\left[ Z_J^{p}\mathbf{1}_{\Omf} \right]&\leq &\frac{1}{E_J^{p}} \int_{0}^{\varepsilon E_J} 4p \ x^{2p-1}e^{ -\frac{x^{2}}{2  E_J \left( 1+\varepsilon \right) }}  dx \\
&\leq &4p\left( 1+\varepsilon \right) ^{p}\int_{0}^{+\infty}u^{2p-1}e^{ -\frac{u^{2}}{2}} du \\
&\leq &4p\left( 1+\varepsilon \right)^{p} \int_{0}^{+\infty}\left( 2t\right) ^{p-1}e^{ -t} dt \\
&\leq &2^{p+1}p\left( 1+\varepsilon\right) ^{p}p!,
\end{eqnarray*}%

and

\begin{equation*}
\sum_{J\in m}\E\left[Z_J^{p}\mathbf{1}_{\Omf} \right]  \leq
2^{p+1}p\left( 1+ \varepsilon\right) ^{p}p! |m|\text{.}
\end{equation*}%

Since $p\leq 2^{p-1}$,

\begin{equation*}
\sum_{J\in m}\E\left[ Z_J^{p}\mathbf{1}_{\Omf} \right] \leq
\frac{p!}{2}\times \left[
2^{5}\left(1+\varepsilon\right)^{2}|m|\right] \times \left[
4\left(1+\varepsilon\right) \right]^{p-2}.
\end{equation*}%

We conclude by taking $v=2^{5}\left( 1+\varepsilon\right) ^{2}|m|$
and $c=4\left( 1+\varepsilon\right)$ (see proposition 2.9 of
\cite{Massart} for the definition of the Bernstein's inequality).

\end{proof}
\section{Simulations and application}\label{Appli}

In the context of RNA-seq experiments, an important question is the
(re)-annotation of the genome, that is, the precise localisation of
the transcribed regions on the chromosomes. In an ideal situation,
when considering the number of reads starting at each position, one
would expect to observe a uniform coverage over each gene
(proportional to its expression level), separated by regions of null
signal (corresponding to non-transcribed regions of the genome). In
practice however, those experiments tend to return very noisy
signals that are best modelled by the negative binomial
distribution. \\

In this Section, we first study the performance of the proposed
penalized criterion by comparing it with others model selection
criteria on a resampling dataset (Subsection \ref{Simul}). Then we
provide an application on real data (Subsection
\ref{ApplcationRNA-SEQ}). Since the penalty depends on the partition
only through its size, the segmentation procedure is two-steps:
first we estimate, for all number of segments $K$ between $1$ and
$K_{max}$, the optimal partition with $K$ segments (i.e. construct
the collection of estimators $\{\hat{s}_K\}_{1\leq K\leq K_{max}}$
where $\hat{s}_K=\arg\min_{\hatsm, m\in \M_K}\{\gamma(\hatsm) \}$).
The optimal solution is obtained using a fast segmentation algorithm
such as the Pruned Dynamic Programming Algorithm (PDPA,
\cite{rigaill_pruned_2010}) implemented for the Poisson and negative
binomial losses or contrasts in the R package
\texttt{Segmentor3IsBack} \cite{PDPA_implementation}. Then, we
choose $K$ using our penalty function which requires the calibration
of the constant $\beta$ that can be tuned according to the data by
using the slope heuristic (see \cite{BMMinPen,Arl_Mas-pente}). Using
the negative binomial distribution requires the knowledge of
parameter $\phi$. We propose to estimate it using a modified version
of the Jonhson and Kotz's estimator \cite{jonhson_kotz}.

\subsection{Simulation study} \label{Simul}

We have assessed the performances of the proposed method (called Penalized PDPA) on a
simulation scenario by comparing to five other procedures both its
choice in the number of segments and the quality of the obtained
segmentation using the Rand-Index $\mathcal{I}$. This index is
defined as follows: let $C_t$ be the true index of the segment to
which base $t$ belongs and let $\hat{C}_t$ be the corresponding
estimated index, then
$$\mathcal{I}= \dfrac{2\sum_{t>s} \left[ \mathbf{1}_{C_t = C_s}
  \mathbf{1}_{\hat{C}_t = \hat{C}_s} +\mathbf{1}_{C_t \neq
    C_s}\mathbf{1}_{\hat{C}_t \neq \hat{C}_s}\right]}{(n-1)(n-2)}. $$

The characteristics of the different algorithms are described in
Table \ref{algo}.

\begin{table}[h]
\begin{small}
\centering
\begin{tabular}{c|c|c|c|c|c|c|}
Algorithm & Dist & Complexity & Inference & Pen & Exact & Reference \\
\hline
Penalized PDPA & NB & $n\log n$ & frequentist & external & exact & {\footnotesize \cite{PDPA_implementation}}\\
PDPA with BIC &  NB & $n\log n$ & frequentist & external & exact & {\footnotesize \cite{PDPA_implementation}}\\
Penalized PDPA & P & $n\log n$ & frequentist & external & exact & {\footnotesize \cite{PDPA_implementation}}\\
PDPA with BIC &  P & $n\log n$ & frequentist & external & exact & {\footnotesize \cite{PDPA_implementation}}\\
PELT with BIC & P & $n$ & frequentist & internal & exact & {\footnotesize \cite{package_changepoint}}\\
CART with BIC & P & $n\log n$ & frequentist & external & heuristic & {\footnotesize \cite{breiman_cart}}\\
postCP with ICL & NB & $n$ & frequentist & external &  exact & {\footnotesize \cite{luong2012fast}}\\
EBS with ICL & NB & $n^2$ & Bayesian & external & exact & {\footnotesize \cite{rigaill_exact_2011}}\\
\end{tabular}
\end{small}
\caption{{\em Properties of segmentation algorithms.} The first
column indicates the name of the algorithm and the criterion used
for the choice of $K$. In the second column, NB stands for the
negative binomial distribution and P for Poisson. The time of each
algorithm is given (column "Complexity") and column "Exact" precises
if the exact solution is reached.} \label{algo}
\end{table}

The data we considered comes from a resampling procedure using real
RNA-seq data. The original data, from a study by the Sherlock
Genomics laboratory at Stanford University, is publicly available on
the NCBIs Sequence Read Archive (SRA, url:
http://www.ncbi.nlm.nih.gov/sra) with the accession number
SRA048710. We created an artificial gene, inspired from the
\textit{Drosophila} inr-a gene, resulting in a $14$-segment signal
with unregular intensities mimicking a differentially transcribed
gene. $100$ datasets are thus created. Results are presented using
boxplots in Figure \ref{droso}. Because PELT's estimate of $K$
averaged around $427$ segments, we did not show its corresponding
boxplot.
\newline

\begin{figure}[h]
\begin{center}
\includegraphics[width=10cm]{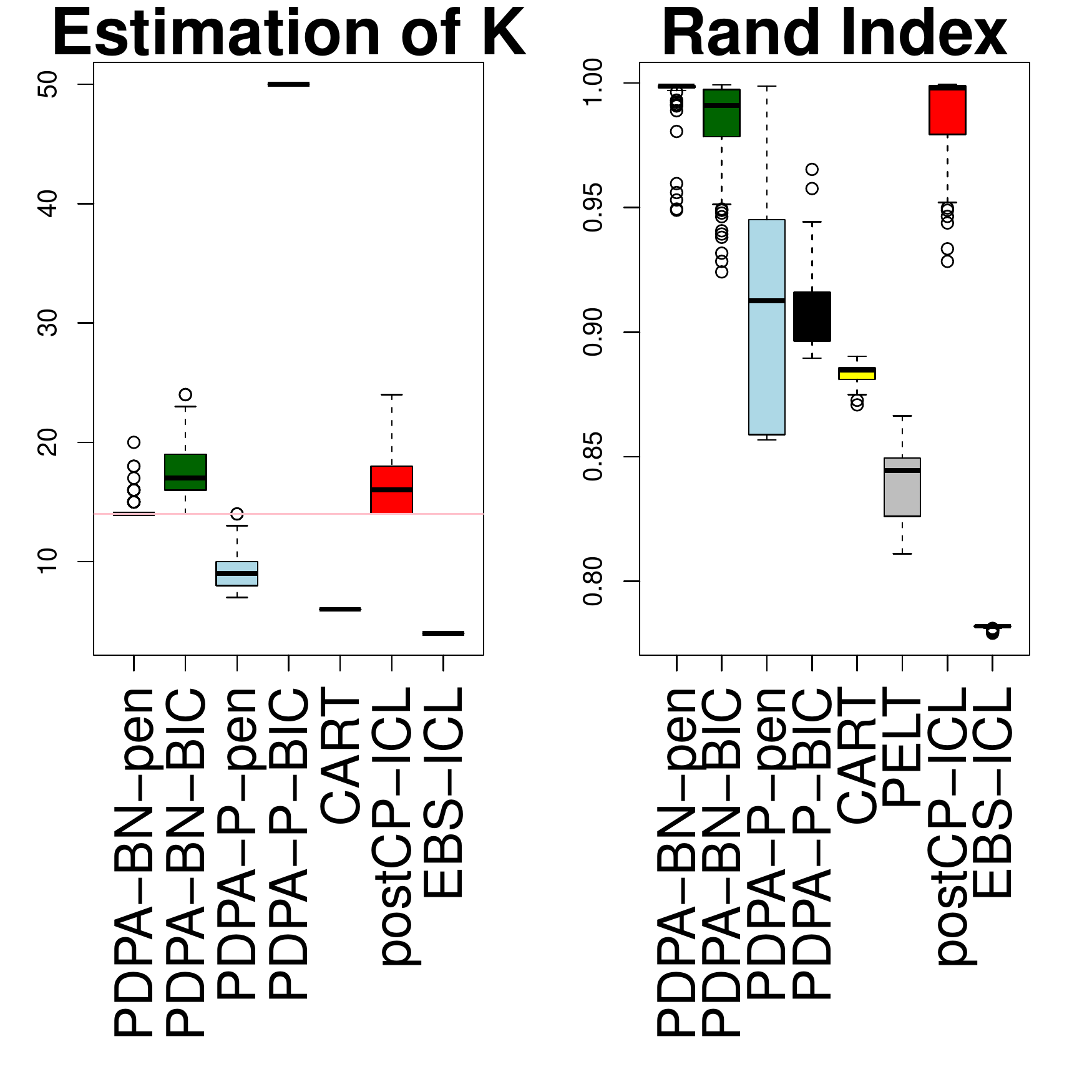}\label{droso}
\end{center}
\caption{{\bf Estimation of K on resampled datasets.}  Left: boxplot of the estimation of $K$ on data-sets simulated by resampling on artificial gene Inr-a. PELT's estimates average at $427$ segments and are not shown. The pink horizontal line indicates the true value of $K$. Right: boxplot of the Rand-Index values for the proposed estimators.}
\end{figure}

We can see that with the negative binomial distribution, not only do
we perfectly recover the true number of segments, but our procedure
outperforms all other approaches. Moreover, the impressive results
in terms of Rand-Index prove that our choice of number of segments
also leads to the almost perfect recovery of the true segmentation.
However, the use of the Poisson loss leads to a constant underestimation of the number of segments, which is reflected on the Rand-Index values. This is due to the inappropriate choice of distribution
(confirmed by the other algorithms implemented for the Poisson loss
which perform worse than the others). It however underlines the need for the development of
methods for the negative binomial distribution. Moreover, in terms
of computational time, the fast algorithm \cite{PDPA_implementation}
is in $\mathcal{O}(n\log n)$, allowing its use on long signals (such
as a whole-genome analysis), even though it is not as fast as CART
or PELT.

\subsection{Segmentation of RNA-Seq data} \label{ApplcationRNA-SEQ}

We apply our proposed procedure for segmenting chromosome $1$ of the
S. Cerevisiae (yeast)  using RNA-Seq data from the Sherlock
Laboratory at Stanford University \cite{Risso_norma} and publicly
available from the NCBI's Sequence Read Archive (SRA,
url:http://www.ncbi.nlm.nih.gov/sra, accession number SRA048710).
An existing annotation is available on the Saccharomyces Genome
Database (SGD) at url:http://www.yeastgenome.org, which allows us
to validate our results. The two distributions (Poisson and negative binomial) are considered here to show the difference. \\

In the Poisson distribution case, we select $106$ segments of which
only $19$ are related to the SGD annotation. Indeed, as illustrated
by Figure \ref{Poisson}, $36$ of the segments have a length smaller
than $10$: the Poisson loss is note adapted to this kind of data
with high variability and it tends to select outliers as segment. On
the contrary, we select $103$ segments in the negative binomial case
most of which (all but $3$) surround known genes from the SGD.  Figure \ref{NegBin}
illustrates the result. However, almost none of those change-points
correspond exactly to annotated boundaries. Discussion with
biologists has increased our belief in the need for genome
(re-)annotation using RNA-seq data, and in  the validity of our
approach.

\begin{figure}[h]
\begin{center}
\includegraphics[width=14cm]{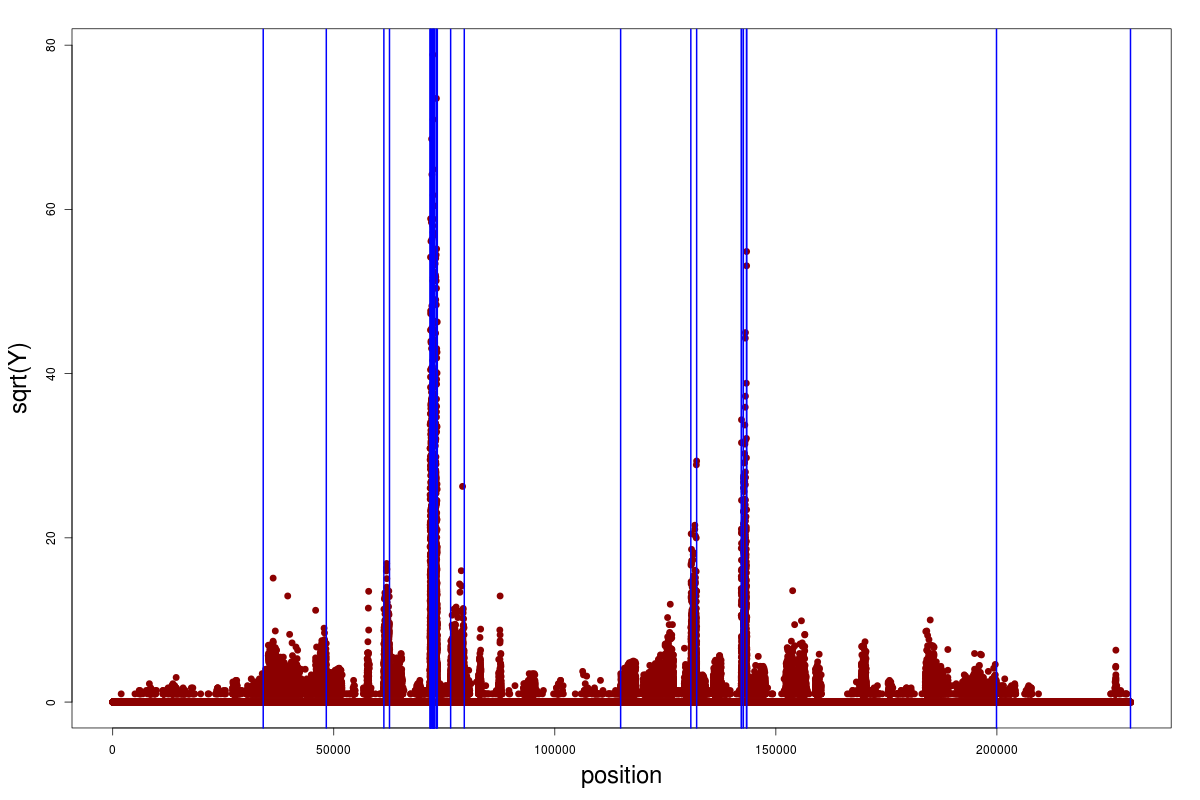}
\end{center}
\caption{\textit{Segmentation of the yeast chromosome 1 using
Poisson loss.} Read-count are represented on a root-squared scale.
The model selection procedure chooses $K=106$
segments.}\label{Poisson}
\end{figure}

\begin{figure}[h]
\begin{center}
\includegraphics[width=14cm]{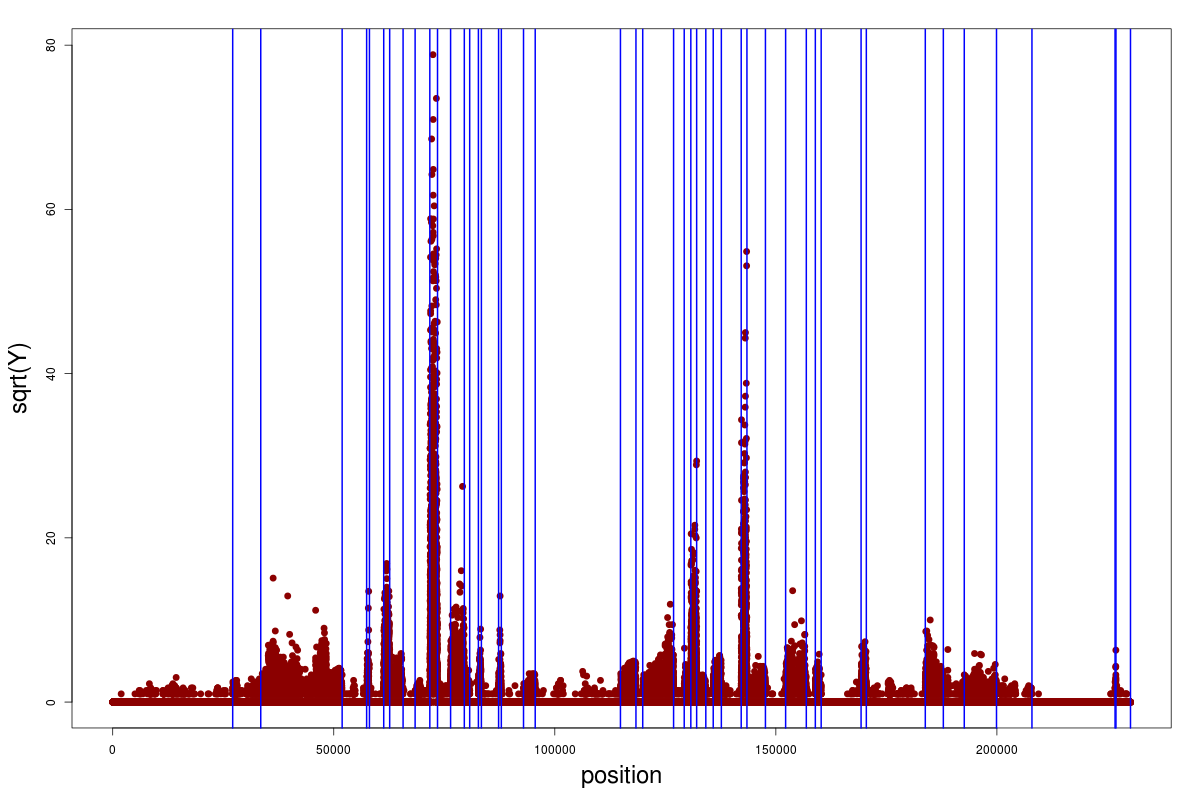}
\end{center}
\caption{\textit{Segmentation of the yeast chromosome 1 using the negative binomial loss.} The model selection procedure chooses $K=103$ segments, most of which surround genes given by the SGD annotation.}\label{NegBin}
\end{figure}

\section{Proof of Theorem \ref{theo} } \label{dem-theo}

Recall that we want to control the three terms in the decomposition
given by (\ref{decomposition}). All the proofs of the different
propositions are given in Section \ref{Appendices}.

    \begin{itemize}
    \renewcommand{\labelitemi}{$\bullet$}
    \item The control the term  $\bar{\gamma}(\hat{s}_{m'})-\bar{\gamma}(\bar{s}_{m'})$ is obtained with the following
    proposition where the set $\OO1$ is defined by
    $$
    \OO1 = \bigcap_{m' \in \M_n}\left\{ \chi^2_{m'} \mathbf{1}_{\Omf} \leq |m'| +8(1+\varepsilon) \sqrt{(L_{m'} |m'|+\xi)|m'|}+ 4(1+\varepsilon)(L_{m'}|m'|+\xi)\right\}.
    $$
    \begin{prpstn} \label{t3}
Let $m'$ be a partition of $\M_n$. Then
      \begin{eqnarray*}
        \left(\bar{\gamma}(\hat{s}_{m'})-\bar{\gamma}(\bar{s}_{m'})\right) \mathbf{1}_{\Omf\cap\OO1} & \leq & C(\varepsilon) \left[|m'| +8(1+\varepsilon) \sqrt{(L_{m'} |m'|+\xi)|m'|} \right.\\
        & & + \left. 4(1+\varepsilon)(L_{m'}|m'|+\xi)  \right] + \dfrac{1}{1+\varepsilon}K(\bar{s}_{m'},\hat{s}_{m'}),
      \end{eqnarray*}
      with $C(\varepsilon)=\frac{1}{2} \left ( \frac{1+\varepsilon}{1-\varepsilon} \right )$ in the Poisson case and $C(\varepsilon)=\frac{1+\varepsilon}{4}$ in the negative binomial case.
    \end{prpstn}

    \vspace{1cm}
  \item The control of the term $\barga(\barsm)-\barga(s)$, or more precisely its expectation, is given by the following proposition:
    \begin{prpstn} \label{t2}
      \begin{eqnarray} \label{terme2}
        |\E[(\barga(\barsm)-\barga(s))\mathbf{1}_{\Omf}  ]| &\leq& \frac{C(\phi,\Gamma,\rho_{min},\rho_{max},\varepsilon,a)}{n^{(a-1)/2}}.
      \end{eqnarray}
    \end{prpstn}

    \vspace{1cm}
  \item To control $\barga(s) - \barga(\bar{s}_{m'})$, we use the following proposition which gives an exponential bound for $\barga(s)-\barga(u)$.

    \begin{prpstn}    \label{s-u}
    Let $s$ and $u$ be two distributions of a sequence $Y$. Let $\gamma$ be the log-likelihood contrast, $\barga(u)=\gamma (u)-\E[\gamma (u)]$, and $K(s,u) $ and $h^2(s,u)$ be respectively the Kullback-Leibler and the squared Hellinger distances between distributions $s$ and $u$. Then $\forall x>0$,
      \begin{eqnarray*}
        \P\left[\bar{\gamma}(s)-\bar{\gamma}(u) \geq K(s,u) - 2h^2(s,u)+2x \right] &\leq &  e^{-x}.
      \end{eqnarray*}
    \end{prpstn}

 Applying it to $u=\bar{s}_{m'}$ yields:
    \begin{eqnarray} \label{terme1}
      \P\left[\bar{\gamma}(s)-\bar{\gamma}(\bar{s}_{m'}) \geq K(s,\bar{s}_{m'}) - 2h^2(s,\bar{s}_{m'})+2x \right] &\leq &  e^{-x}.
    \end{eqnarray}
    We then define $\mO2=  \bigcap_{m' \in \M_n}\left\{ \bar{\gamma}(s)-\bar{\gamma}(\bar{s}_{m'}) \leq K(s,\bar{s}_{m'}) - 2h^2(s,\bar{s}_{m'})+2(L_{m'}|m'|+\xi) \right\}$.
    \end{itemize}

    \vspace{0.5cm}
Let $\Omega(\varepsilon,\xi) = \Omf\cap\OO1 \cap \mO2$. Then,
combining equation (\ref{terme1}) and proposition \ref{t3},  we get
for $m' = \hat{m}$,

    \begin{eqnarray*}
      (\barga(\barsm) -\barga(\hatsmh))\O & = & (\barga(s)-\barga(\barsmh))\O +(\barga(\barsm) -\barga(s))\O +(\barga(\barsmh)-\barga(\hatsmh))\O\\
&\leq & \left[K(s,\barsmh)-2h^2(s,\barsmh)\right]\O + R\O + \dfrac{1}{1+\varepsilon}K(\barsmh,\hatsmh)\O \\
      &  & +C(\varepsilon) \left[|\mh| +8(1+\varepsilon) \sqrt{(L_{\mh} |\mh|+\xi)|\mh|}+ 4(1+\varepsilon)(L_{\mh}|\mh|+\xi) \right] \\
      &&+ 2 L_{\mh} |\mh| +2\xi,
    \end{eqnarray*}
     with $R= \bar{\gamma}(\bar{s}_m)-\bar{\gamma}(s)$. So that
    \begin{eqnarray*}
      K(s,\hat{s}_{\hat{m}})\O &\leq & \left[K(s,\barsmh)-2h^2(s,\barsmh)\right]\O + \dfrac{1}{1+\varepsilon}K(\barsmh,\hatsmh)\O \\
      & & + C(\varepsilon)\left[|\mh| +8(1+\varepsilon) \sqrt{(L_{\mh} |\mh|+\xi)|\mh|}+ 4(1+\varepsilon)(L_{\mh}|\mh|+\xi) \right]\\
      & & + K(s,\barsm)\O + 2L_{\mh} |\mh|+2\xi +R\O -pen(\hat{m})+pen(m).
    \end{eqnarray*}

    And since \begin{itemize}
      \item  $K(s,\hat{s}_{\hat{m}}) =  K(s,\bar{s}_{\hat{m}}) +
      K(\barsmh,\hatsmh)$ (see equation (\ref{Pythagore})),
      \item $K(s,u)\geq 2h^2(s,u)$ (see lemma 7.23 in \cite{Massart}),
      \item $h^2(s,\hatsmh)\leq 2\left(h^2(s,\barsmh)+h^2(\barsmh , \hatsmh)
      \right)$ (using inequality $2 a b \leq \kappa a^2+\kappa^{-1} b^2$ with
      $\kappa=1$),
    \end{itemize}

    \begin{eqnarray*}
      \dfrac{\varepsilon}{1+\varepsilon} h^2(s,\hatsmh)\O &\leq& K(s,\barsm)\O + R\O -pen(\hat{m})+pen(m)\\
      & & + |\mh| C(\varepsilon)\left[1+ (1+\varepsilon)\left(8\sqrt{L_{\mh}}+ \varepsilon +4 L_{\mh}\right) \right] + 2 L_{\mh}|\mh|\\
      & & + 2\xi\left[1 + C(\varepsilon)\left(8(1+\varepsilon)\dfrac{2}{\varepsilon} +4(1+\varepsilon) \right)\right].
    \end{eqnarray*}

    But
    \begin{eqnarray*}
 C(\varepsilon)\left[1+ (1+\varepsilon)\left(8\sqrt{L_{\mh}}+ \varepsilon +4 L_{\mh}\right) \right] + 2 L_{\mh}& \leq &C(\varepsilon)\left[1+ (1+\varepsilon)\left(\varepsilon + 8\sqrt{L_{\mh}}+ 8 L_{\mh}\right) \right]\\
&  \leq & C_2(\varepsilon)\left[1 +8\sqrt{L_{\mh}}+ 8 L_{\mh}
\right].
    \end{eqnarray*}
with
$C_2(\varepsilon)=\dfrac{1}{2}\left(\dfrac{1+\varepsilon}{1-\varepsilon}\right)^3$
for $(\mathcal{P})$ and
$C_2(\varepsilon)=\dfrac{1}{4}\left(1+\varepsilon \right)^3$ for
$(\NB)$. So we have
    \begin{eqnarray*}
      \dfrac{\varepsilon}{1+\varepsilon} h^2(s,\hatsmh)\O &\leq& K(s,\barsm)\O + R\O -pen(\hat{m})+pen(m)\\
      & & + |\mh| C_2(\varepsilon) \left(1 + 4\sqrt{L_{\mh}} \right)^2 + 2\xi\left[1 +(1+\varepsilon)C(\varepsilon)\left(\dfrac{8}{\varepsilon}+2
\right)\right].
    \end{eqnarray*}

    By assumption, $pen(\mh)\geq \beta |\mh| \left(1 + 4\sqrt{L_{\mh}} \right)^2$.
    Choosing  $\beta = C_2(\varepsilon)$ yields

    \begin{eqnarray*}
      h^2(s,\hat{s}_{\hat{m}})\O &\leq& C_{\beta} \left[K(s,\barsm)\O + R\O + pen(m)\right]+\xi C(\beta).
    \end{eqnarray*}

Then, using propositions \ref{t2} and \ref{t3}, we have
$\P\left(\OO1^C \right) \leq \sum_{m' \in \M_n} e^{-L_{m'}|m'| +
\xi}$
    and $\P\left(\mO2^C \right) \leq \sum_{m' \in \M_n} e^{-L_{m'}|m'| +
    \xi}$. So that using hypothesis (\ref{weights}),
    \begin{eqnarray*}
      \P\left(\OO1^C \cup \mO2^C \right) &\leq& 2 \sum_{m' \in \M_n} e^{-L_{m'}|m'| + \xi} \leq 2\Sigma e^{-\xi},
    \end{eqnarray*}
    and thus $ \P\left(\OO1 \cap \mO2 \right) \geq 1- 2\Sigma e^{-\xi}$. We now integrate over $\xi$, and using equation (\ref{terme2}), we
    get with a probability larger than $1- 2\Sigma e^{-\xi}$

    \begin{eqnarray*}
      \E\left[h^2(s,\hatsmh)\mathbf{1}_{\Omf} \right] &\leq & C_{\beta}  \left[K(s,\barsm)+ \frac{C(\phi,\Gamma,\rho_{min},\rho_{max},\beta,a)}{n^{(a-1)/2}}+pen(m)\right] + \Sigma C(\beta).
    \end{eqnarray*}

    And since  $ \E\left[h^2(s,\hatsmh)\mathbf{1}_{\Omf^C} \right] \leq \dfrac{C(\phi,\Gamma,\rho_{min},\rho_{max},\beta,a)}{n^{a-1}}$, we have
    \begin{eqnarray*}
      \E\left[h^2(s,\hatsmh) \right] &\leq & C_{\beta} \left[K(s,\barsm) +pen(m)\right] + C'(\phi,\Gamma,\rho_{min},\rho_{max},\beta,\Sigma).
    \end{eqnarray*}

    Finally, by minimizing over $m\in\mathcal{M}_n$, we get
    \begin{eqnarray*}
      \E\left[h^2(s,\hatsmh) \right] &\leq & C_{\beta} \inf_{m\in\mathcal{M}_n} \left\{K(s,\barsm) +pen(m)\right\} + C'(\phi,\Gamma,\rho_{min},\rho_{max},\beta,\Sigma).\\
    \end{eqnarray*}

\section{Appendices} \label{Appendices}

\subsection{Proof of proposition \ref{Risk-control}} \label{risk}
Using Pythagore-type identity,  we obtain the following
decomposition (see for example \cite{Cast00}):
\begin{eqnarray} \label{Pythagore}
  K(s,\hatsm) = K(s,\barsm) + K(\barsm,\hatsm).
\end{eqnarray}
The objective is then to obtain a lower bound of
$\E[K(\barsm,\hatsm)]$ in the two considered distribution cases.

\subsubsection*{Poisson case}
We have
    \begin{eqnarray*}
      K(\bar{s}_m,\hat{s}_m)=\sum_{J \in m}|J|\left(\bar{Y}_J-\bar{\lambda}_J-\bar{\lambda}_J\log\dfrac{\bar{Y}_J}{\bar{\lambda}_J}\right)= \sum_{J\in m}|J| \bar{\lambda}_J \Phi \left( \log{\dfrac{\bar{Y}_J}{\bar{\lambda}_J}} \right).
    \end{eqnarray*}
    where $\Phi(x)=e^x-1-x$. Since $\dfrac{1}{2} x^2 (1 \wedge e^x) \leq \Phi(x) \leq \dfrac{1}{2} x^2
(1 \vee e^x)$, then  on  $\Omf$, we have

    \begin{eqnarray*}
      \dfrac{1}{2} \log^2{\dfrac{\bar{Y}_J}{\bar{\lambda}_J}} \left(1 \wedge \dfrac{\bar{Y}_J}{\bar{\lambda}_J}\right) &\leq& \Phi\left(\log{\frac{\bar{Y}_J}{\bar{\lambda}_J}}\right)  \leq  \dfrac{1}{2} \log^2{\dfrac{\bar{Y}_J}{\bar{\lambda}_J}} \left(1 \vee \dfrac{\bar{Y}_J}{\bar{\lambda}_J}\right), \\
\frac{1-\varepsilon}{2} \log^2{\frac{\bar{Y}_J}{\bar{\lambda}_J}} &
\leq &  \Phi\left(\log{\frac{\bar{Y}_J}{\bar{\lambda}_J}}\right)
\leq  \frac{1+\varepsilon}{2}
\log^2{\frac{\bar{Y}_J}{\bar{\lambda}_J}}.
    \end{eqnarray*}

    So
    \begin{eqnarray}    \label{lienKV2}
      \frac{1-\varepsilon}{2} V^2_m &\leq &K(\bar{s}_m,\hat{s}_m)\leq \frac{1+\varepsilon}{2}
      V^2_m,
    \end{eqnarray}

    where
    \begin{eqnarray}    \label{def-V2}
      V^2_m=V^2(\bar{s}_m,\hat{s}_m)&=&\sum_{J\in m} |J| \bar{\lambda}_J \log^2{\frac{\bar{Y}_J}{\bar{\lambda}_J}}=\sum_{J\in m} |J| \dfrac{(\bar{Y}_J-\bar{\lambda}_J)^2}{\bar{\lambda}_J} \left( \dfrac{\log\frac{\bar{Y}_J}{\bar{\lambda}_J}}{\frac{\bar{Y}_J}{\bar{\lambda}_J}-1}\right)^2.
    \end{eqnarray}

    And using, for $x>0$, $ \dfrac{1}{1\vee x} \leq \dfrac{\log x}{x-1}\leq \dfrac{1}{1 \wedge x} $, we get, on $\Omf$
    \begin{eqnarray} \label{link:V2Chi2}
      \dfrac{1}{(1+\varepsilon)^2}\ \chi^2_m \leq V^2_m \leq \dfrac{1}{(1-\varepsilon)^2}\ \chi^2_m.
    \end{eqnarray}
    So
    \begin{eqnarray*}
      \dfrac{1-\varepsilon}{2(1+\varepsilon)^2}\ \chi^2_m \mathbf{1}_{\Omf}\leq K(\bar{s}_m,\hat{s}_m)\mathbf{1}_{\Omf} \leq \dfrac{1+\varepsilon}{2(1-\varepsilon)^2}\ \chi^2_m\mathbf{1}_{\Omf}.
    \end{eqnarray*}
    On one hand,  $\E\left[\chi^2_m\right] = |m|$, and
    \begin{eqnarray*}
      \dfrac{1-\varepsilon}{2(1+\varepsilon)^2} |m| - \E\left[\chi^2_m\mathbf{1}_{\Omf^C} \right]  \leq \E\left[K(\bar{s}_m,\hat{s}_m)\mathbf{1}_{\Omf}\right] \leq \dfrac{1+\varepsilon}{2(1-\varepsilon)^2} |m|.
    \end{eqnarray*}

Since $\chi^2_m \leq \frac{1}{ \Gamma (\log{(n)})^2
\rho_{min}}\sum_{J\in m}(Y_J-\lambda_J)^2 \leq \frac{1}{\Gamma
(\log{(n)})^2
\rho_{min}}\left(\sum_{t}Y_t-\sum_t\lambda_t\right)^2$, using
Cauchy-Schwarz Inequality, we get

\begin{eqnarray*}
\E\left[\chi^2_m\mathbf{1}_{\Omf^C} \right]
&\leq&\frac{1}{\Gamma (\log{(n)})^2 \rho_{min}}\left[3\left(\sum_t \lambda_t\right)^2+\sum_t \lambda_t \right]^{1/2} P(\Omf^C)^{1/2} \\
&\leq& C(\Gamma, \rho_{min},\rho_{max}) \frac{n}{(\log{(n)})^2}P(\Omf^C)^{1/2} \\
&\leq& C(\Gamma, \rho_{min},\rho_{max}) n^{\alpha} P(\Omf^C)^{1/2} \\
&\leq&\frac{C(\phi,\Gamma,\rho_{min},\rho_{max},\varepsilon,a)}{n^{a/2-\alpha}},\\
\end{eqnarray*}
where $\alpha = 1-2 \frac{\log{(\log{(n)})}}{\log{(n)}}$, $n \geq
2$. For example, $\alpha=0.62$ for $n=10^6$.

    \vspace{0.5cm}

On the other hand, using $\log{1/x} \geq 1-x$ for all $x>0$,
$\E\left[K(\bar{s}_m,\hat{s}_m)\mathbf{1}_{\Omf^C}\right] \geq 0$.
Finally, we have
    \begin{eqnarray*}
      K(s,\bar{s}_m)+\dfrac{1-\varepsilon}{2(1+\varepsilon)^2}|m|- \dfrac{C_1(\Gamma,\rho_{min},\rho_{max},\varepsilon,a)}{n^{a/2-\alpha}}\leq \E[K(s,\hat{s}_m)],
    \end{eqnarray*}

\subsubsection*{Negative binomial case}

We have $K(\barsm, \hatsm) = \phi \sum_{J \in m}\dfrac{|J|}{p_J}
h_{\frac{\phi}{\phi+\bar{Y}_J}}\left(p_J\right)$ and $\forall
0<a<1,\, h_a(x)\geq \dfrac{1-x}{1-a}\log^2\left(\dfrac{1-x}{1-a}
\right)$. \\
Then on $\Omf$
\begin{eqnarray*}
K(\barsm, \hatsm) &\geq& \phi \sum_{J \in m}\dfrac{|J|}{p_J}
\dfrac{1-p_J}{\frac{\bar{Y}_J}{\phi+\bar{Y}_J}} \log^2
\left(\dfrac{\frac{\bar{Y}_J}{\phi+\bar{Y}_J}}{1-p_J} \right).
\end{eqnarray*}
Introducing
\begin{eqnarray} \label{defV2BN}
 V^2_m = \sum_{J \in m}\phi |J| \dfrac{1-p_J}{p_J} \log^2 \left(\dfrac{\frac{\bar{Y}_J}{\phi+\bar{Y}_J}}{1-p_J} \right),
\end{eqnarray}
we get
\begin{eqnarray} \label{KV2BN}
K(\barsm, \hatsm) \geq V^2_m,
\end{eqnarray}
and since
$\bar{Y}_J-\phi\frac{1-p_J}{p_J}=\frac{\phi+\bar{Y}_J}{p_J}\left(\frac{\bar{Y}_J}{\phi+\bar{Y}_J}-(1-p_J)
\right)$, we have
\begin{eqnarray*}
 V^2_m &=&\sum_{J \in m}  |J| \left(\frac{\phi}{\phi+\bar{Y}_J}\right)^2 \dfrac{\left(\bar{Y}_J-\phi\frac{1-p_J}{p_J}\right)^2 }{\phi\frac{1-p_J}{p_J}}  \left[\dfrac{\log \left(\dfrac{\frac{\bar{Y}_J}{\phi+\bar{Y}_J}}{1-p_J} \right)}{\dfrac{\frac{\bar{Y}_J}{\phi+\bar{Y}_J}}{1-p_J} -1}
 \right]^2.
 \end{eqnarray*}
 And finally,
 \begin{eqnarray*}
 K(\barsm, \hatsm) \mathbf{1}_{\Omf} &\geq& \rho_{min}^2 \dfrac{(1-\varepsilon)^2}{(1+\varepsilon)^4}\ \chi^2_m
 \mathbf{1}_{\Omf}.
\end{eqnarray*}
Moreover, on one hand we have $|m|\leq \E\left[\chi^2_m\right] \leq
\frac{1}{\rho_{min}}|m|$. On the other hand, since $\chi^2_m  \leq
\frac{1}{\Gamma (\log{(n)})^2 \phi
(1-\rho_{max})}\left(\sum_{t}Y_t-\sum_t E_t\right)^2$, using
Cauchy-Schwarz Inequality, we get

\begin{eqnarray*}
\E\left[\chi^2_m\mathbf{1}_{\Omf^C} \right] &\leq&\frac{\left[\E \left(Y_t-E_t \right)^4 +6\phi^2\sum_{(t,l),l\neq t}\frac{1-p_t}{p_t^2} \frac{1-p_l}{p_l^2}\right]^{1/2}}{\Gamma (\log{(n)})^2 \phi(1-\rho_{max})} P(\Omf^C)^{1/2}, \\
&\leq& C(\Gamma, \rho_{min},\rho_{max}) n^{\alpha} P(\Omf^C)^{1/2}, \\
&\leq&\frac{C(\phi,\Gamma,\rho_{min},\rho_{max},\varepsilon,a)}{n^{a/2-\alpha}},\\
\end{eqnarray*}
where $\alpha = 1-2 \frac{\log{(\log{(n)})}}{\log{(n)}}$, $n \geq
2$. Finally, we have

    \begin{eqnarray*}
      K(s,\bar{s}_m)+\rho_{min}^2\dfrac{(1-\varepsilon)^2}{(1+\varepsilon)^4}|m|- \dfrac{C(\phi,\Gamma,\rho_{min},\rho_{max},\varepsilon,a)}{n^{a/2-\alpha}}\leq
      \E[K(s,\hat{s}_m)].
    \end{eqnarray*}

\subsection{Proof of proposition \ref{t3}} \label{a-t3}

\subsubsection*{Poisson case}

The term to be controlled is
$\bar{\gamma}(\hat{s}_{m'})-\bar{\gamma}(\bar{s}_{m'})=\sum_{J \in
m'}|J|\left(\bar{Y}_J-\bar{\lambda}_J \right) \log
\dfrac{\bar{Y}_J}{\bar{\lambda}_J}$. Using Cauchy-Schwarz
inequality, we have

    \begin{eqnarray*}
      \bar{\gamma}(\bar{s}_{m'})-\bar{\gamma}(\hat{s}_{m'}) &\leq & \sqrt{\chi^2_{m'}}\ \  \sqrt{V^2_{m'}},
    \end{eqnarray*}
    with $\chi_{m'}^2$ and $V^2_{m'}$ defined as in equations (\ref{def-chi}) and (\ref{def-V2}). Then, using equation (\ref{lienKV2})

    \begin{eqnarray*}
      \left(\bar{\gamma}(\bar{s}_{m'})-\bar{\gamma}(\hat{s}_{m'})\right)\mathbf{1}_{\Omf} &\leq & \sqrt{\chi^2_{m'}}\ \  \sqrt{\frac{2}{1-\varepsilon}K(\bar{s}_{m'},\hat{s}_{m'})},
    \end{eqnarray*}
    and using $2 a b \leq \kappa a^2+\kappa^{-1} b^2$ for all $\kappa>0$, we get

    \begin{eqnarray} \label{cont-Emilie}
      \left(\bar{\gamma}(\bar{s}_{m'})-\bar{\gamma}(\hat{s}_{m'})\right)\mathbf{1}_{\Omf} &\leq & \dfrac{\kappa}{2} \chi^2_{m'}+ \frac{ \kappa^{-1}}{1-\varepsilon} K(\bar{s}_{m'},\hat{s}_{m'}).
    \end{eqnarray}
And with proposition \ref{cont-chi2}, we get, for $\kappa =
\dfrac{1+\varepsilon}{1-\varepsilon}= 2C(\varepsilon)$,
    \begin{eqnarray*}
  & &    \left(\bar{\gamma}(\hat{s}_{m'})-\bar{\gamma}(\bar{s}_{m'})\right) \mathbf{1}_{\Omf\cap\OO1} \\
  & & \leq \dfrac{1+\varepsilon}{2(1-\varepsilon)} \left[|m'|+8(1+\varepsilon) \sqrt{(L_{m'} |m'|+\xi)|m'|}+ 4(1+\varepsilon)(L_{m'}|m'|+\xi)  \right] +\dfrac{1}{1+\varepsilon}K(\bar{s}_{m'},\hat{s}_{m'}).
    \end{eqnarray*}

\subsubsection*{Negative binomial case}

In this case we can write  $\barga(\hatsmp)-\barga(\barsmp)= \sum_{J
\in m'} |J| \left(\bar{Y}_J -\bar{E}_J \right)
\log\dfrac{\frac{\bar{Y}_J }{\phi+\bar{Y}_J }}{1-p_J}$. Again, using
Cauchy-Schwarz inequality, and with $\chi^2_m$ and $V^2_m$ defined
by equations (\ref{def-chi}) and (\ref{defV2BN}), we get
    \begin{eqnarray*}
      \bar{\gamma}(\bar{s}_{m'})-\bar{\gamma}(\hat{s}_{m'}) &\leq & \sqrt{\chi^2_{m'}}\ \  \sqrt{V^2_{m'}},
    \end{eqnarray*}
    so that with equation (\ref{KV2BN}) and  $2 a b \leq \kappa a^2+\kappa^{-1} b^2$ for all $\kappa>0$
    \begin{eqnarray}
      \left(\bar{\gamma}(\bar{s}_{m'})-\bar{\gamma}(\hat{s}_{m'})\right)\mathbf{1}_{\Omf} &\leq & \dfrac{\kappa}{2} \chi^2_{m'}+ \frac{\kappa^{-1}}{2} K(\bar{s}_{m'},\hat{s}_{m'}).
    \end{eqnarray}

Finally, with proposition \ref{cont-chi2} and  $\kappa =
\dfrac{1+\varepsilon}{2} = 2C(\varepsilon)   $,
    \begin{eqnarray*}
  & &    \left(\bar{\gamma}(\hat{s}_{m'})-\bar{\gamma}(\bar{s}_{m'})\right) \mathbf{1}_{\Omf\cap\OO1} \\
  & & \leq  \dfrac{1+\varepsilon}{4} \left[|m'|+8(1+\varepsilon) \sqrt{(L_{m'} |m'|+\xi)|m'|}+ 4(1+\varepsilon)(L_{m'}|m'|+\xi)  \right] +\dfrac{1}{1+\varepsilon}K(\bar{s}_{m'},\hat{s}_{m'}).
    \end{eqnarray*}

\subsection{Proof of proposition \ref{t2}} \label{a-t2}

\subsubsection*{Poisson case}
    Noting that $\E[(\barga(\barsm)-\barga(s))\mathbf{1}_{\Omf} ] = -\E[(\barga(\barsm)-\barga(s))\mathbf{1}_{\Omf^C} ]$,
    we have
    \begin{eqnarray*}
      |\E[(\barga(\barsm)-\barga(s))\mathbf{1}_{\Omf}  ]| &\leq& |\E[(\barga(\barsm)-\barga(s))\mathbf{1}_{\Omf^C} ]| \leq \E[| (\barga(\barsm)-\barga(s)) |\mathbf{1}_{\Omf^C} ] \\
& \leq& \E\left[\left| \left(\sum_J \sum_t (Y_t-E_t) \log{(\rho_{max}/\rho_{min})}\right) \right|\mathbf{1}_{\Omf^C} \right] \\
& \leq & \log{(\rho_{max}/\rho_{min})} \times \E\left[\left|  \sum_t (Y_t-E_t) \right|\mathbf{1}_{\Omf^C} \right] \\
& \leq& \log{(\rho_{max}/\rho_{min})} \times \left (\left[\E\left(\sum_t  (Y_t-E_t)\right)^2   \right]^{1/2} \times \left(P(\Omf^C\right)^{1/2} \right ) \\
& \leq& \left(n \rho_{max}\right)^{1/2} \times \log{(\rho_{max}/\rho_{min})} \times (P(\Omf^C)^{1/2},
    \end{eqnarray*}
which concludes the proof.

\subsubsection*{Negative binomial case}
Once again, $\E[(\barga(\barsm)-\barga(s))\mathbf{1}_{\Omf} ] =
-\E[(\barga(\barsm)-\barga(s))\mathbf{1}_{\Omf^C} ]$, and
    \begin{eqnarray*}
      |\E[(\barga(\barsm)-\barga(s))\mathbf{1}_{\Omf}  ]| &\leq& |\E[(\barga(\barsm)-\barga(s))\mathbf{1}_{\Omf^C} ]| \leq \E[| (\barga(\barsm)-\barga(s)) |\mathbf{1}_{\Omf^C} ] \\
& \leq& \E\left[\left| \left(\sum_J \sum_t \left(Y_t-\phi\frac{1-p_t}{p_t}\right) \log{(1/(1-\rho_{min}))}\right) \right|\mathbf{1}_{\Omf^C} \right] \\
& \leq & \log{(1/(1-\rho_{min}))} \times \E\left[\left| \sum_t \left(Y_t-E_t\right) \right|\mathbf{1}_{\Omf^C} \right] \\
& \leq& \left(n \phi\frac{1}{\rho_{min}^2}\right)^{1/2} \times \log{\frac{1}{1-\rho_{min}}} \times (P(\Omf^C)^{1/2} \\
    \end{eqnarray*}
which concludes the proof.

\subsection{Proof of proposition \ref{s-u}} \label{concentration}

Using the Markov inequality $\P\left[\bar{\gamma}(s)-\bar{\gamma}(u)
\geq b \right] \leq \inf_a \left[e^{-ab}
\E\left(e^{a(\bar{\gamma}(s)-\bar{\gamma}(u))} \right) \right]$ with
$a=\frac{1}{2}$, we get
 \begin{eqnarray*}
     \P\left[\bar{\gamma}(s)-\bar{\gamma}(u) \geq b \right] &\leq &\exp\left[ -\frac{b}{2} +\log \E\left[  \exp\left(\frac{1}{2}\left(\gamma(s)-\gamma(u)\right) +\frac{1}{2} \E\left [\gamma(u)-\gamma(s)\right ]\right)\right]\right] \\
     &\leq &\exp\left[ -\frac{b}{2} +\frac{1}{2}K(s,u) +\log\E\left[\exp\left(-\frac{1}{2}\sum_t \log \P_s(X_t=Y_t)+\log \P_u(X_t=Y_t) \right)\right]\right] \\
      &\leq &\exp\left[ -\frac{b}{2} +\frac{1}{2}K(s,u) +\sum_t \log\E\sqrt{\frac{\P_u(X_t=Y_t)}{\P_s(X_t=Y_t)}} \right] \\
      &\leq &\exp\left[ -\frac{b}{2} +\frac{1}{2}K(s,u) +\sum_t \E\sqrt{\frac{\P_u(X_t=Y_t)}{\P_s(X_t=Y_t)}} -n \right] \\
      &\leq &\exp\left[ -\frac{b}{2} +\frac{1}{2}K(s,u) -h^2(s,u) \right] \\
 \end{eqnarray*}
where $\P_s=\P$ denote the probability under the distribution $s$.
Thus
      \begin{eqnarray*}
        \P\left[\bar{\gamma}(s)-\bar{\gamma}(u) \geq K(s,u)-2h^2(s,u)+2x \right] &\leq &  e^{-x}.
      \end{eqnarray*}

\begin{acknowledgement}
The authors wish to thank St\'ephane Robin for more than helpful
discussions on the statistical aspect and Gavin Sherlock for his
insight on the biological applications.
\end{acknowledgement}

\bibliographystyle{bmc_article}
\bibliography{Biblio}
\end{document}